\documentclass[12pt]{amsart}
\usepackage[mathscr]{eucal}
\usepackage{amssymb, amsfonts, ytableau, comment, color}

\theoremstyle{plain} 
\newtheorem{thm}{Theorem}[section]
\newtheorem{cor}[thm]{Corollary}
\newtheorem{lem}[thm]{Lemma}

\theoremstyle{definition}

\theoremstyle{remark}

\numberwithin{equation}{section}

\def\NN{{\mathbb N}}
\def\ZZ{{\mathbb Z}}

\def\fS{{\mathfrak S}}
\def\fm{{\mathfrak m}}

\def\ba{{\mathbf a}}
\def\bb{{\mathbf b}}

\def\be{{\mathbf e}}

\def\height{\operatorname{ht}}

\def\too{\longrightarrow}

\def\reg{\operatorname{reg}}

\def\SYT{\operatorname{SYT}}
\def\Tab{\operatorname{Tab}}
\def\supp{\operatorname{supp}}
\def\init{\operatorname{in}}
\def\trm{\operatorname{trm}}

\def\ISp{I^{\rm Sp}}
\def\JSp{J^{\rm Sp}}

\def\ba{\mathbf a}

\def\too{\longrightarrow}

\def\bxi{\overline{x_n^i}}
\def\bxm{\overline{x_n^m}}

\def\chara{\operatorname{char}}

\def\<{{\langle}}
\def\>{{\rangle}}

\title{Regularity of Cohen-Macaulay Specht ideals}

\author{Kosuke Shibata}
\address{Department of Mathematics, Okayama University, Okayama, Okayama 700-8530, Japan}
\email{pfel97d6@s.okayama-u.ac.jp}
\author{Kohji Yanagawa}
\address{Department of Mathematics, Kansai University, Suita, Osaka 564-8680, Japan}
\email{yanagawa@kansai-u.ac.jp}
\thanks{The second author is partially supported by JSPS Grant-in-Aid for Scientific Research (C) 19K03456.}
\date{--, --, --}

\begin{document}
\maketitle

\begin{abstract}
For a partition $\lambda$ of $n \in \NN$, let $I^{\rm Sp}_\lambda$ be the ideal of $R=K[x_1,\ldots,x_n]$ generated by all Specht polynomials of shape $\lambda$.
In the previous paper, the second author showed that if $R/I^{\rm Sp}_\lambda$ is Cohen-Macaulay, then $\lambda$ is either $(n-d,1,\ldots,1),(n-d,d)$, or $(d,d,1)$, and the converse is true if ${\rm char}(K)=0$. 
In this paper, we compute the Hilbert series of $R/I^{\rm Sp}_\lambda$ for $\lambda=(n-d,d)$ or $(d,d,1)$. Hence, we get the Castelnuovo-Mumford regularity of $R/I^{\rm Sp}_\lambda$, when it is Cohen-Macaulay. In particular, $\ISp_{(d,d,1)}$ has a $(d+2)$-linear resolution in the Cohen--Macaulay case.  

\end{abstract}

\section{Introduction}
For a positive integer $n$,  a {\it partition} of $n$ is a sequence $\lambda = (\lambda_1, \ldots, \lambda_l)$ of integers with $\lambda_1 \ge \lambda_2 \ge \cdots \ge \lambda_l \ge 1$ and $\sum_{i=1}^l \lambda_i =n$.  
The {\it Young tableau} of shape $\lambda$ is a bijection from $[n]:=\{1,2, \ldots, n\}$ to the set of boxes in the Young diagram of $\lambda$. 
 For example, the following is a tableau of shape  $(4,2,1)$. 
\begin{equation}\label{ex of T}
\begin{ytableau}
3 & 5 & 1& 7   \\
6 & 2    \\
4 \\
\end{ytableau}
\end{equation}
Let $\Tab(\lambda)$ be the set of  Young tableaux of shape $\lambda$. If $\lambda = (\lambda_1, \ldots, \lambda_l)$, then we simply write as $\Tab(\lambda_1, \ldots, \lambda_l)$. 
We say a tableau $T$ is {\it standard}, if all columns (resp. rows) are increasing from top to bottom (resp. from left to right). 
Let $\SYT(\lambda)$ (or $\SYT(\lambda_1, \ldots, \lambda_l)$) be the set of standard tableaux of  shape $\lambda = (\lambda_1, \ldots, \lambda_l)$. 

Let $R=K[x_1, \ldots, x_n]$ be a polynomial ring over a field $K$, $\lambda$ a partition of $n$, and $T$ a Young tableau of shape $\lambda$.   
If the $j$-th column of $T$ consists of $j_1, j_2, \ldots, j_m$ in the order from top to bottom, then 
$$f_T (j) := \prod_{1 \le s < t \le m} (x_{j_s}-x_{j_t}) \in R$$
(if the $j$-th column has only one box, then we set $f_T(j)=1$).  
 The {\it Specht polynomial} $f_T$ of $T$ is given by  
$$f_T := \prod_{j=1}^{\lambda_1} f_T(j).$$
For example, if $T$ is the tableau \eqref{ex of T}, then $f_T=(x_3-x_6)(x_3-x_4)(x_6-x_4)(x_5-x_2).$ 

The symmetric group $\fS_n$ acts on the vector space $V_\lambda$ spanned by 
$\{ \, f_T  \mid T \in \Tab(\lambda) \}.$
 An $\fS_n$-module of this form is called a {\it Specht module}, and very important in the  theory of  symmetric groups, especially in the characteristic 0 case.   See, for example, \cite{Sa}. Here, we only remark that $\{ f_T  \mid T \in \SYT(\lambda) \}$ forms a basis of $V_\lambda$. 

In the previous paper \cite{Y},  the second author studied the {\it ideal}  
$$\ISp_\lambda :=(\, f_T  \mid T \in \Tab(\lambda) )$$
of $R$.  We have $\height(\ISp_\lambda)=\lambda_1$ by \cite[Proposition~2.3]{Y}. 
The main result of \cite{Y} states the following.

\begin{thm}[{\cite[Proposition~2.8 and Corollary~4.4]{Y}}]
\label{prev paper main}
If $R/\ISp_\lambda$ is Cohen--Macaulay, then one of the following conditions holds.  
\begin{itemize}
\item[(1)] $\lambda =(n-d, 1, \ldots, 1)$, 
\item[(2)] $\lambda =(n-d,d)$, 
\item[(3)] $\lambda =(d,d,1)$. 
\end{itemize}
If $\chara(K)=0$, the converse is also true. 
\end{thm}

The case (1) is treated in the joint paper \cite{WY} with J. Watanabe, and it is shown that $R/\ISp_{(n-d,1,\ldots,1)}$ is Cohen--Macaulay over any $K$.
To prove the last assertion of the above theorem for the cases (2) and (3), we first show that $\ISp_\lambda$ is a radical ideal (at least, in these cases) over any $K$, and use a result of Etingof et al. \cite{EGL}, which concerns the characteristic 0 case.

The paper \cite{WY} computes the Betti numbers of $R/\ISp_{(n-d, 1, \ldots, 1)}$, it means that we know its Hilbert series in this case.  
 In the present paper,  we compute the Hilbert series 
$$H(R/\ISp_\lambda, t):=\sum_{i \in \NN} \dim_K[R/\ISp_\lambda]_i \cdot t^i $$ in the cases (2) and (3) of Theorem~\ref{prev paper main}. 
The main tool for computation is the following recursive formulas
$$H(R/\ISp_{(n-d,d)}, t) = H(S/\ISp_{(n-d-1,d)}, t) + \frac{t}{1-t} H(S/\ISp_{(n-d,d-1)}, t)$$
for $n-d >d \ge 2$, and 
$$H(R/\ISp_{(d,d)}, t) = H(S/\ISp_{(d-1,d-1,1)}, t) + \frac{t}{1-t} H(S/\ISp_{(d,d-1)}, t)$$
for $n=2d \ge 4$.  Here we set $S=K[x_1, \ldots, x_{n-1}]$.  Since $(n-d-1,d)$ is a partition of $n-1$, $\ISp_{(n-d-1,d)}$ is an ideal of $S$. The same is true for other partitions of $n-1$. 

As an application, we have the following. 

\begin{thm}\label{main cor}
If $R/\ISp_{(n-d,d)}$ is Cohen-Macaulay (e.g., when $\chara(K)=0$), then we have $\reg (R/\ISp_{(n-d,d)})=d$ for $d \ge 2$, and if $R/\ISp_{(d,d,1)}$ is Cohen-Macaulay (e.g., when $\chara(K)=0$), then $\reg (R/\ISp_{(d,d,1)})=d+1$.  Hence $\ISp_{(d,d,1)}$ has a $(d+2)$-linear resolution in this case. 
\end{thm}


Since $R/\ISp_{(d,d,1)}$ (not $S/\ISp_{(d-1,d-1,1)}$) does not appear in the above recursion formulas, these formulas are not enough. So we use  \cite[Theorem 3.2]{RV} for $R/\ISp_{(d,d,1)}$. 
However, this result assumes the Cohen--Macaulay property, so we have to show that 
the Hilbert seres of $R/\ISp_{(d,d,1)}$ does not depend on $\chara(K)$. We prove this (essentially) in \S3 using the Gr\"obner basis argument.   


\section{Main theorem and related arguments}
For the definition and basic properties of Specht ideals $\ISp_\lambda$, consult the previous section. 
Here we just remark that the Cohen--Macaulay-ness of $R/\ISp_\lambda$ actually depends on $\chara(K)$. For example, 
$R/\ISp_{(n-3,3)}$ is Cohen--Macaulay if and only if $\chara(K) \ne 2$. The same is true for $R/\ISp_{(2,2,1)}$. See \cite[Theorem~5.3]{Y}.     
If $\chara(K) =2$, {\it Macaulay2} computation shows that  $R/\ISp_{(2,2,1)}$ and $R/\ISp_{(n-3,3)}$ for $n \le 10$ do not satisfy even Serre's $(S_2)$ condition. So the $(S_2)$ condition of $R/\ISp_\lambda$ also depends on $\chara(K)$ (recall that the $(S_2)$-ness of the Stanley--Reisner  ring $K[\Delta]$ does not depend on $\chara(K)$). 
 {\it Macaulay2} computation also shows that  $R/\ISp_{(4,4)}$  is not Cohen--Macaulay, if $\chara(K)=2,3$. See \cite[Conjecture~5.5]{Y}. To the authors' best knowledge, examples of $R/\ISp_\lambda$ satisfying the $(S_2)$ condition are Cohen--Macaulay.  

\bigskip

We regard $S=K[x_1,\ldots, x_{n-1}]$ as a subring of $R=K[x_1,\ldots, x_n]$. 
If $\mu$ is a partition of $n-1$, then the Specht ideal $\ISp_\mu$ is an ideal of $S$.  

\begin{thm}\label{main}
The Hilbert series of $R/\ISp_{(n-d,d)}$ is given by 
$$H(R/\ISp_{(n-d,d)}, t)= \frac{1+h_1 t +h_2 t^2+\cdots + h_d t^d }{(1-t)^d}$$
with 
$$
h_i = \begin{cases}
\binom{n-d+i-1}{i} & \text{if $1 \le i \le d-1$,} \\
\binom{n-1}{d-2} &  \text{if $ i = d$.}
\end{cases}
$$
Similarly, when $n=2d+1$, the Hilbert series of $R/\ISp_{(d,d,1)}$ is given by 
$$H(R/\ISp_{(d,d,1)}, t)= \frac{1+h_1 t +h_2 t^2+\cdots + h_{d+1} t^{d+1} }{(1-t)^{d+1}}$$
with 
$$h_i = \binom{d+i-1}{i}$$
for all $1 \le i \le d+1$. 
\end{thm}



\noindent{\it The proof of Theorem~\ref{main cor} (assuming Theorem~\ref{main}).} 
For a Cohen--Macaulay graded ring $R/I$ of dimension $d$ whose Hilbert series is given by 
$$H(R/I, t)= \frac{1+h_1 t +h_2 t^2+\cdots + h_st^s }{(1-t)^d}$$
with $h_s \ne 0$, it is well-known that $\reg (R/I)=s$.  So the assertion follows from  Theorem~\ref{main}. 
\qed

\section{The initial monomials of Specht polynomials}
In this section, we will give Gr\"obner basis theoretic results, which can be used to show the Hilbert series do not depend $\chara(K)$. 
See \cite[\S 15]{E} for notions and results of the Gr\"obner basis theory. 
Here we consider the lexicographic order on $R$ with  $x_n \succ x_{n-1} \succ \cdots \succ x_1$. 
Let $\init(f)$ be the initial monomial of $0 \ne f \in R$.  

Consider a tableau 
$$
T = 
\ytableausetup
{mathmode, boxsize=2.2em}
\begin{ytableau}
i_1 & i_2  & \cdots &i_d & i_{d+1} & \cdots & i_{n-d}  \\
j_1 & j_2 & \cdots & j_d
\end{ytableau}
\in \Tab(n-d, d).
$$
Since the permutation of $i_k$ and $j_k$ only changes the sign of $f_T$, we may assume that $i_k < j_k$ for all  $1 \le k \le d$. 
Then we have $\init(f_T)= x_{j_1}x_{j_2} \cdots x_{j_d}$.

The following lemma holds for a  general partition $\lambda$, and must be well-known to specialists. 
Since we could not find appropriate references, we give a quick proof for the reader's convenience.     

\begin{lem}\label{initial terms} For a partition $\lambda =(n-d,d)$, we have the following.  

\begin{itemize}
\item[(1)] For distinct  $T, T'\in \SYT(\lambda)$, we have $\init(f_T) \ne \init(f_{T'})$.  

\item[(2)] Let $V_{\lambda} \subset R$ be the Specht module of shape $\lambda$. For $0 \ne f \in V_\lambda$, there is a unique $T \in \SYT(\lambda)$ such that $\init(f)=\init(f_T)$. 
\end{itemize}
\end{lem}

\begin{proof}
(1) If $\init(f_T) =\init(f_{T'})$ holds for $T, T'\in \SYT(\lambda)$, then the second rows of $T$ and $T'$ are same. It means that $T=T'$. 

(2) It is well-known that $\{ f_T \mid T \in \SYT(\lambda) \}$ forms a basis of $V_\lambda$. Hence the assertion follows from (1). 
\end{proof}

In the rest of this section, we assume that $n=2d$. 
Let $\fm^{\<d+1 \>}$ be the ideal of $R$ generated by all squarefree monomials of degree $d+1$, and set 
$$\JSp_{(d,d)}:=\ISp_{(d,d)} +\fm^{\<d+1\>}.$$ 
For $\ba \in \NN^n$, set $x^\ba:=\prod_{i \in [n]} x_i^{a_i} \in R$. 
For $f=\sum_{\ba \in \NN^n} c_\ba x^\ba \in R$ ($c_\ba \in K$), we call
$$\trm(f):= \sum_{x^\ba \not \in \fm^{\<d+1\>}} c_\ba x^\ba \in R$$
the {\it trimmed form} of $f$. For example, if $d=2$ and $f=x_1x_4^2-2x_2x_3^2+3x_1x_3x_4-x_2x_3x_4$, then we have  $\trm(f)=x_1x_4^2-2x_2x_3^2$.

For $F \subset [n]$ with $\# F=:c \le d$,  let 
$\Tab_F (d, d-c)$ be the set of Young tableaux of shape  $(d, d-c)$ with the letter set $[n] \setminus F$. 
For example, if $n=8$ (i.e., $d=4$) and $F=\{1,6\}$, then 
$$
\ytableausetup
{mathmode, boxsize=1.5em}
\begin{ytableau}
2  & 8 & 7& 3\\
4 & 5
\end{ytableau}
$$
 is an element of $\Tab_F(4,2)$. For the convention, set $\Tab_\emptyset (d, d):= \Tab(d,d)$. If $\#F=d$, then $T \in \Tab_F(d,0)$ consists of a single row, and we have $f_T =1$.  

For a subset $F \subset [n]$, set $x^{F} := \prod_{i \in F}x_i \in R$ and $x^{2F} := \prod_{i \in F}x_i^2 \in R$. 
For a monomial $x^\ba \in R$, set $\supp(x^\ba):=\{ \,  i \mid a_i > 0 \}$. 

\begin{lem}\label{trm f_T}
Let $x^\ba \in R$ be a monomial with $F:= \supp (x^\ba)$, and set $c:= \# F$.  
If $\trm(x^\ba f_T) \ne 0$ for $T \in \Tab(d,d)$, then we have $c \le d$, and there is some $T' \in \Tab_F(d,d-c)$ such that 
$$\trm(x^\ba f_T) =x^{\ba} x^F  f_{T'}.$$ 
In prticular, we have 
$$\trm(x^F f_T) = x^{2F}  f_{T'}.$$ 
The converse also holds, that is, any  $x^{\ba} x^F f_{T'}$ for $T' \in \Tab_F(d,d-c)$ equals $\trm(x^\ba f_T)$ for some $T \in \Tab(d,d)$.  
\end{lem}

\begin{proof}
Without loss of generality, we may assume that  $F =\{1, 2, \ldots, c \}$.  
Note that nonzero terms of $f_T$  are ($\pm$ of) squarefree monomials of degree $d$. Hence if  $\trm(x^\ba f_T) \ne 0$, then $T$ is of the form 
$$
\ytableausetup
{mathmode, boxsize=3.2em}
\begin{ytableau}
i_1 &  \cdots &i_{d-c} & i_{d-c+1} & i_{d-c+2} &\cdots  & i_d  \\
j_1 & \cdots & j_{d-c} & 1 & 2 & \cdots & c
\end{ytableau}
$$
after a suitable column permutation and permutations of the two boxes in the same columns  ($f_T$ is stable under these permutations up to sign).  Moreover,  
$$T':=
\ytableausetup
{mathmode, boxsize=3.2em}
\begin{ytableau}
i_1 &  \cdots &i_{d-c} & i_{d-c+1} & i_{d-c+2} &\cdots  & i_d  \\
j_1 & \cdots & j_{d-c} 
\end{ytableau}
$$
satisfies the expected condition. 

The last assertion can be proved in a similar way.    
\end{proof}

\begin{thm}\label{Grobner of JSp}
With the above situation, 
\begin{equation}\label{Grobner}
 \Biggl( \bigcup_{F \subset [n] } \{  \, \trm(x^{F} f_T) \mid T \in \Tab (d, d) \}  \setminus \{ 0\} \Biggr ) \cup G(\fm^{\<d+1\>})
\end{equation}
forms a Gr\"obner basis of $\JSp_{(d,d)}$. Here $G(\fm^{\<d+1\>})$ is the set of squarefree monomials of degree $d+1$. 
\end{thm}

\begin{proof}
Take $F_1, F_2 \subset [n]$ and  $T_1, T_2 \in \Tab(d, d)$ with $\varphi_1 := \trm(x^{F_1} f_{T_1}) \ne 0$ and  $\varphi_2 := \trm(x^{F_2} f_{T_2}) \ne 0$. 
We have some $\ba, \bb \in \NN^n$ such that the least common multiple of $\init(\varphi_1)$ and  $\init(\varphi_2 )$ 
coincides with $x^\ba \init(\varphi_1) = x^\bb \init(\varphi_2)$.  Note that $x^\ba$ and  $x^\bb$ need not be squarefree, and  $x^\ba \init(\varphi_1) = x^\bb \init(\varphi_2)$ might belong to $\fm^{\<d+1\>}$. These phenomena make the following argument a bit complicated.

We set $\psi := x^\ba \varphi_1 \pm x^\bb \varphi_2$, where we take $\pm$ to cancel the initial terms. By Buchberger's criterion (\cite[Theorem~15.8]{E}),  it suffices to show that $\psi$ can be reduced to 0 modulo  \eqref{Grobner} by the division algorithm. To do this, it suffices to show that $\trm(\psi)$  can be reduced to 0 modulo 
\begin{equation}\label{Grobner small}
 \bigcup_{F \subset [n] } \{  \, \trm(x^{F} f_T) \mid T \in \Tab (d, d) \}  \setminus \{ 0\}. 
\end{equation}

If $\trm(\psi) \ne 0$,  then at least one of the following conditions holds 
\begin{itemize}
\item[(i)] $\trm(x^\ba \varphi_1) \ne 0$, equivalently, for  $G_1 := \supp(x^\ba x^{F_1})$, $x^{G_1}$ divides some non-zero term of $f_{T_1}$. 

\item[(ii)]  $\trm(x^\bb \varphi_2) \ne 0$, equivalently, for  $G_2 := \supp(x^\bb x^{F_2})$, $x^{G_2}$ divides some non-zero term of $f_{T_2}$. 
\end{itemize}

Assume that only (i)  holds. 
Since (ii) is not satisfied now, we have 
$$\trm(\psi)=\trm(x^\ba \varphi_1)=\trm(x^\ba x^{F_1} f_{T_1})=  x^\ba x^{F_1} x^{G_1}f_{T_1'}$$ for some $T_1' \in \Tab_{G_1}(d,d-c_1)$ by Lemma~\ref{trm f_T}, where $c_1:= \# G_1$. Hence $\trm(\psi)$ belongs to 
$$\<\,   \trm(x^\ba x^{F_1} f_T) \mid T \in \Tab(d,d) \, \>=x^\ba x^{F_1} x^{G_1}\<\,    f_{T'} \mid {T'} \in \Tab_{G_1}(d,d-c_1) \, \> .$$
The subset 
\begin{equation}\label{G_1} 
\{  \, \trm(x^{G_1} f_T) \mid T \in \Tab (d, d) \, \}  \setminus \{ 0\}
\end{equation}
of  \eqref{Grobner small} spans the subspace 
$$V_1:=\< \, \trm(x^{G_1} f_{T}) \mid T \in \Tab(d,d) \, \>=  x^{2G_1} \< \,  f_{T'} \mid T' \in \Tab_{G_1}(d,d-c_1) \, \>,$$
which is actually the Specht module of shape $(d, d-c_1)$. In fact, the symmetric group $\fS_{[n] \setminus G_1}$ acts on $V_1$. 
The multiplication $\times (x^\ba x^{F_1}/x^{G_1})$ gives a bijection from $V_1$ to 
$x^\ba x^{F_1} x^{G_1}\<\,    f_{T'} \mid {T'} \in \Tab_{G_1}(d,d-c_1) \, \>$, to which $\trm(\psi)$ belongs, and this bijection preserves the monomial order.  
Hence $\trm(\psi)$ can be reduced to 0 modulo \eqref{G_1} by Lemma~\ref{initial terms} (2).
The case when only (ii)  holds can be proved in the same way.

Next consider the case when both (i) and (ii) are satisfied. Set 
$$V_1 := \<\,   \trm(x^\ba x^{F_1} f_T) \mid T \in \Tab(d,d) \,  \, \> \ \text{and} \  
V_2 := \<\,   \trm(x^\bb x^{F_2} f_T) \mid T \in \Tab(d,d) \,  \, \>. $$
Clearly, $\trm(x^\ba \varphi_1) \in V_1$ and $\trm(x^\bb \varphi_2) \in V_2$. 
If a monomial $x^\be$ appears as a non-zero term of $f \in V_1$ (resp.  $f \in V_2$), then we have $\prod_{e_i \ge 2} x_i^{e_i-1}= x^\ba x^{F_1}$ (resp. $\prod_{e_i \ge 2} x_i^{e_i-1}= x^\bb x^{F_2}$).  
 Hence if  $x^\ba x^{F_1} \ne  x^\bb x^{F_2}$, then  $V_1 \cap V_2 =\{ 0\}$. Therefore, either  $V_1 = V_2$ or $V_1 \cap V_2 =\{ 0\}$ holds. 
If $V_1=V_2$, then we have $\trm(x^\ba \varphi_1) , \trm(x^\bb \varphi_2) \in V_1$, and hence $\trm(\psi) \in V_1$. 
So the situation is essentially the same as the previous cases, and $\trm(\psi)$  is reduced to 0 modulo \eqref{G_1}. If $V_1 \cap V_2 =\{ 0\}$,  then we have 
$$\trm(\psi) = \trm(x^\ba \varphi_1) \pm \trm(x^\bb \varphi_2) \in V_1+ V_2=V_1 \oplus V_2,$$
and  no terms of $\trm(x^\ba \varphi_1) $ and $\trm(x^\bb \varphi_2) $ are canceled. Hence $\trm(\psi) $ can  be reduced to 0 modulo the subset 
$$ (\{  \, \trm(x^{G_1} f_T) \mid T \in \Tab (d, d) \}  \cup \{  \, \trm(x^{G_2} f_T) \mid T \in \Tab (d, d) \}  ) \setminus \{ 0\}$$
of \eqref{Grobner small}. 
In fact, the ``$V_1$-part" and the ``$V_2$-part" can be reduced to 0 individually. 
\end{proof}

\begin{cor}\label{Hilb Jdd char free}
The Hilbert function of $\JSp_{(d,d)}$ does not depend on $\chara(K)$. 
\end{cor}

\begin{proof}
In general, a homogeneous ideal $I \subset R$ and its initial ideal $\init(I) :=(\init(f) \mid 0 \ne f \in I)$ have the same Hilbert function, and the Hilbert function of a monomial ideal (e.g.,  $\init(I)$) does not depend on $\chara(K)$.   Since the Gr\"obner basis of $\JSp_{(d,d)}$ does not depend on $\chara(K)$ by Theorem~\ref{Grobner of JSp}, the assertion follows. 
\end{proof}

\section{The proof of the main theorem}
To prove Theorem~\ref{main}, we can extend the base field. So we assume that $\# K=\infty$ in this section.

For a subset $\emptyset \ne F \subset [n]$, set $P_F := (x_i -x_j \mid i,j \in F) \subset R$. Clearly, this is a prime ideal with $\height (P_F) =\# F -1$. 
By \cite[Theorem~3.1]{Y}, $\ISp_{(n-d,d)}$ is a radical ideal, and hence we have 
\begin{equation}\label{(n-d,d) components}
\ISp_{(n-d,d)} =\bigcap_{\substack{F \subset [n] \\ \# F =n-d+1}} P_F
\end{equation}
by \cite[Proposition~2.4]{Y}. 

\begin{lem}[\cite{Y}]\label{(n-d,d) meaning}
For $f \in R$, the following are equivalent. 
\begin{itemize}
\item[(1)] $f \in \ISp_{(n-d,d)}$. 
\item[(2)] Take $\ba=(a_1, \ldots, a_n) \in K^n$. If there is a subset $F \subset [n]$ with $\# F = n-d+1$ such that $a_i=a_j$ for all $i, j \in F$, then we have $f(\ba)=0$.  
\end{itemize}
\end{lem}

\begin{proof}
Easily follows from \eqref{(n-d,d) components}. 
\end{proof}

\begin{lem}\label{(n-d,d) S-module}
If  $n-d > d \ge 2$, then we have 
$$\ISp_{(n-d,d)} \cap S = \ISp_{(n-d-1,d)}.$$
Moreover, as graded $S$-modules, we have 
\begin{equation}\label{(n-d,d) decomposition}
(R/\ISp_{(n-d,d)})/(S/\ISp_{(n-d-1,d)}) \cong \bigoplus_{m \ge 1} (S/\ISp_{(n-d,d-1)})(-m).
\end{equation}
\end{lem}

\begin{proof}
For a subset $F \subset [n-1]$, we set $P'_F := (x_i -x_j \mid i,j \in F) \subset S$. 
It is easy to check that $P_F \cap S =P'_{F \setminus \{n\} }$ for $F \subset [n]$.  
For  $F \subset [n-1]$ and $i \in F$, set $F':= (F \setminus \{ i\}) \cup \{n\}$. Then we have $\height(P_F) = \height(P_{F'})$ but $(P_{F'} \cap S) \subsetneq  (P_F\cap S)$, unless $F=\{i\}$ (in this case, $P_F =(0)$). Hence we have 
\begin{eqnarray*}
\ISp_{(n-d,d)} \cap S &=& \Biggl( \bigcap_{\substack{F \subset [n] \\ \# F =n-d+1}} P_F   \Biggr) \cap S\\
 &=& \bigcap_{\substack{F \subset [n] \\ \# F =n-d+1}} (P_F  \cap S) \\
 &=& \bigcap_{\substack{F \subset [n], n \in F  \\ \# F =n-d+1}} (P_F  \cap S) \\
&=& \bigcap_{\substack{F \subset [n-1]  \\ \# F =n-d}} P'_F \\
&=& \ISp_{(n-d-1,d)}. 
\end{eqnarray*}
Hence we have shown the  first assertion, and we see that $S/\ISp_{(n-d-1,d)}$ can be seen as  an $S$-submodule of $R/\ISp_{(n-d,d)}$ in the natural way. To show the second assertion, set $M:=(R/\ISp_{(n-d,d)})/(S/\ISp_{(n-d-1,d)})$. This is a graded $S$-module. 

We denote the image of $x_n^m \in R$ in $M$ for $m \ge 1$ by $\bxm$.  First, we show that $\ISp_{(n-d,d-1)}  \subset (0:_M \bxm)$. It suffices to show that $f_T \bxm =0$, equivalently, $f_T x_n^m \in S+ \ISp_{(n-d,d)} $ for all 
$$
T = 
\ytableausetup
{mathmode, boxsize=2em}
\begin{ytableau}
i_1 & i_2  & \cdots &i_{d-1} & i_d & \cdots & i_{n-d}  \\
j_1 & j_2 & \cdots & j_{d-1}
\end{ytableau}
\in \Tab(n-d, d-1).
$$
We can prove this by induction on $m$. In fact, if we set  
$$
T' = 
\ytableausetup
{mathmode, boxsize=2em}
\begin{ytableau}
i_1 & i_2  
& \cdots &i_{d-1} & i_d & i_{d+1}&\cdots & i_{n-d}  \\
j_1 & j_2 
& \cdots & j_{d-1} & n
\end{ytableau}
\in \Tab(n-d, d),
$$
we have $f_{T'} = f_T \cdot (x_{i_d}-x_n)$, and hence $f_T x_n= f_T x_{i_d}-f_{T'} \in  S+ \ISp_{(n-d,d)}$  (note that $x_{i_d} \in S$). If $m >1$, then 
$$ f_T x_n^m = (f_T x_n)x_n^{m-1}= ( f_T x_{i_d}-f_{T'} )x_n^{m-1}=(f_T x_n^{m-1})x_{i_d}-f_{T'}x_n^{m-1}.$$ 
Here $f_{T'}x_n^{m-1} \in \ISp_{(n-d,d)}$, and $f_T x_n^{m-1} \in S+\ISp_{(n-d,d)}$ by the induction hypothesis.  
Hence $(f_T x_n^{m-1})x_{i_d} \in  S+\ISp_{(n-d,d)}$, and 
$$ f_T x_n^m =(f_T x_n^{m-1})x_{i_d}-f_{T'}x_n^{m-1} \in S+\ISp_{(n-d,d)}.$$

Next we show that 
\begin{equation}\label{direct sum}
M \cong \bigoplus_{m \ge 1} S \cdot \bxm,
\end{equation}
as $S$-modules. Here  $S \cdot \bxm$ is the $S$-submodule generated by $\bxm \in M$, and it does not mean this is a free $S$-module. 
To do this, assume that 
$$\sum_{i=1}^l f_i \bxi =0$$
for some $l \ge 1$ and $f_1, \ldots, f_l \in S$. Then we have 
$$\sum_{i=1}^l f_i x_n^i \in  S+\ISp_{(n-d,d)}.$$
Hence there is some $f_0 \in S$ such that 
$$\sum_{i=0}^l f_i x_n^i \in \ISp_{(n-d,d)}.$$
Take $\ba=(a_1, \ldots, a_{n-1}) \in K^{n-1}$. If there is a subset $F \subset [n-1]$ with $\# F = n-d+1$ such that $a_i=a_j$ for all $i, j \in F$, then we have 
$$\sum_{i=0}^l f_i(\ba) b^i =0$$
for all $b \in K$ by Lemma~\ref{(n-d,d) meaning}.  Since $\# K =\infty$ now, we have 
$$\sum_{i=0}^l f_i(\ba) x_n^i =0.$$
 Hence we have $f_i(\ba)=0$ for all $i$, and $f_i \in \ISp_{(n-d,d-1)}$ by Lemma~\ref{(n-d,d) meaning}. Since we have shown that $\ISp_{(n-d,d-1)} \bxi=0$ for $i \ge 1$,  we get the direct sum \eqref{direct sum}. 
Moreover, the above argument also shows that  $\ISp_{(n-d,d-1)}  \supset (0:_M \bxi)$. Hence we have  $\ISp_{(n-d,d-1)}  = (0:_M \bxi)$, and $S \cdot \bxi \cong S/\ISp_{(n-d,d-1)} $. So we are done.  
\end{proof}

If $n=2d+1$, $\ISp_{(d,d,1)} \subset R$ is a radical ideal by \cite[Theorem~4.2]{Y}, and we have 
\begin{equation}\label{(d,d,1) components}
\ISp_{(d,d,1)} =\bigcap_{\substack{F \subset [n] \\ \# F =d+1}} P_F
\end{equation}
by \cite[Proposition~2.4]{Y}.

\begin{lem}
If  $n=2d \ge 4$, then we have 
$$\ISp_{(d,d)} \cap S = \ISp_{(d-1,d-1,1)}.$$
Moreover, as graded $S$-modules, we have 
\begin{equation}\label{(d,d) decomposition}
(R/\ISp_{(d,d)})/(S/\ISp_{(d-1,d-1,1)}) \cong \bigoplus_{m \ge 1} (S/\ISp_{(d,d-1)})(-m).
\end{equation}
\end{lem}

\begin{proof}
The assertion follows from the argument similar to the proof of Lemma~\ref{(n-d,d) S-module}, while we use \eqref{(d,d,1) components} this time. 
\end{proof}

\begin{cor}\label{Induction}If $n-d >d \geq 2$, then we have
\begin{equation}\label{Hilb (n-d,d)}
H(R/\ISp_{(n-d,d)}, t) = H(S/\ISp_{(n-d-1,d)}, t) + \frac{t}{1-t} H(S/\ISp_{(n-d,d-1)}, t).
\end{equation}
Similarly, if $n=2d$ and $d \ge 2$, we have 
 \begin{equation}\label{Hilb (d,d)}
H(R/\ISp_{(d,d)}, t) = H(S/\ISp_{(d-1,d-1,1)}, t) + \frac{t}{1-t} H(S/\ISp_{(d,d-1)}, t).
\end{equation}
\end{cor}

\begin{proof}
The first assertion follows from \eqref{(n-d,d) decomposition}. In fact, we have 
\begin{eqnarray*}
H(R/\ISp_{(n-d,d)}, t) &=& H(S/\ISp_{(n-d-1,d)}, t) +\sum_{m=1}^\infty t^m \cdot  H(S/\ISp_{(n-d,d-1)}, t)\\
&=& H(S/\ISp_{(n-d-1,d)}, t) + \frac{t}{1-t} H(S/\ISp_{(n-d,d-1)}, t).
\end{eqnarray*}
Similarly, the second assertion follows from  \eqref{(d,d) decomposition}. 
\end{proof}

Now we can start the proof of the main theorem. 

\medskip

\noindent{\it The proof of Theorem~\ref{main}.} 
We prove the assertion  by induction on $n$. Note that we also have to prove that the Hilbert series of $R/\ISp_{(n-d,d)}$ and $R/\ISp_{(d,d,1)}$ do not depend on $\chara(K)$. 

It is easy to see that $R/\ISp_{(n-1,1)} \cong K[X]$, and its Hilbert series is $1/(1-t)$. So the assertion holds in this case. 
Similarly, $R/\ISp_{(1,1,1)}$ is a hypersurface ring of degree 3, and its Hilbert series is $(1+t +t^2)/(1-t)^2$. So the assertion also holds in this case. 

We assume that the statement holds for $n-1$.
If  $n$ is an odd number $2d+1$, we first treat $\ISp_{(d,d,1)}$.  Recall that  $\ISp_{(d,d)} \subset K[x_1, \ldots, x_{n-1}]$, and    let  $\pi : R \to S \, (\cong R/(x_n))$ be the natural surjection. 
Clearly, we have $ S/\pi(\ISp_{(d,d,1)}) \cong R/(\ISp_{(d,d,1)} +(x_n))$. 
As shown in \cite[\S2]{Y}, $x_n$ is $R/\ISp_\lambda$-regular for any non-trivial partition $\lambda$ of $n$. So we can recover the Hilbert series of $R/\ISp_{(d,d,1)}$ from that of 
$ S/\pi(\ISp_{(d,d,1)})$.  

Since $\pi(\ISp_{(d,d,1)}) = \ISp_{(d,d)} \cap \fm^{\<d+1\>}$ by \cite[Lemma~2.10]{Y}, we have the short exact sequence  
\begin{equation}\label{Jdd}
0 \to S/\pi(\ISp_{(d,d,1)}) \too S/\ISp_{(d,d)} \oplus S/\fm^{\<d+1\>} \too 
S/\JSp_{(d,d)} \too 0. 
\end{equation}
The Hilbert series of $S/\ISp_{(d,d)}$ (resp. $S/\JSp_{(d,d)} $) do not depend on $\chara(K)$ by the induction hypothesis (resp. Corollary \ref{Hilb Jdd char free}). 
Since $\fm^{\<d+1\>}$ is a monomial ideal, its Hilbert series is also characteristic free. 
Hence $H(S/\pi(\ISp_{(d,d,1)}),t)$ does not depend on $\chara(K)$ by \eqref{Jdd}. 
So the same is true for $H(R/\ISp_{(d,d,1)},t)$. 

So we may assume that $\chara(K)=0$, then $R/\ISp_{(d,d,1)}$ is Cohen-Macaulay.
The number of minimal generators of $\ISp_{(d,d,1)}$ is $\#\SYT(d,d,1)$, and we have 
\begin{eqnarray*}
\#\SYT(d,d,1)=\frac{(2d+1)!}{(d+2)d!(d+1)(d-1)!}=\binom{2d+1}{d+2}
\end{eqnarray*}
by the hook formula (c.f., \cite[Theorem~3.10.2]{Sa}). 
Since $R/\ISp_{(d,d,1)}$ is Cohen-Macaulay and $\height(\ISp_{(d,d,1)})=d$, 
$\ISp_{(d,d,1)}$ has a $(d+2)$-linear resolution by \cite[Theorem 3.2]{RV}. Hence, 
\begin{eqnarray*}
H(R/\ISp_{(d,d,1)},t)={\displaystyle\frac{\sum_{i=0}^{d+1}\binom{i+d-1}{i}t^i}{(1-t)^{d+1}}}
\end{eqnarray*}
by \cite[Exercises 5.3.16]{V}. So the assertion holds in this case.

Next, we consider the Hilbert series of $R/\ISp_{(n-d,d)}$ with $n-d > d \ge 2$ (without the assumption that $\chara(K)=0$). By the induction hypothesis, $H(S/\ISp_{(n-d-1,d)}, t)$ and $H(S/\ISp_{(n-d,d-1)}, t)$ do not depend on $\chara(K)$, and we have
$$H(S/\ISp_{(n-d-1,d)}, t)= \frac{1+h'_1 t +h'_2 t^2+\cdots + h'_d t^d }{(1-t)^d}$$
with 
$$
h'_i = \begin{cases}
\binom{n-d+i-2}{i} & \text{if $1 \le i \le d-1$,} \\
\binom{n-2}{d-2} &  \text{if $ i = d$,}
\end{cases}
$$
and 
$$H(S/\ISp_{(n-d,d-1)}, t)= \frac{1+h''_1 t +h''_2 t^2+\cdots + h''_{d-1} t^{d-1} }{(1-t)^{d-1}}$$
with 
$$
h''_i = \begin{cases}
\binom{n-d+i-1}{i} & \text{if $1 \le i \le d-2$,} \\
\binom{n-2}{d-3} &  \text{if $ i = d-1$.}
\end{cases}
$$
Since $\dim R/\ISp_{(n-d,d)}=d$, the denominator of the Hilbert series of $H(R/\ISp_{(n-d,d)}, t)$ is $(1-t)^d$, and the numerator is 
$$1+(h'_1+1)t+ (h'_2+h''_1)t^2+\cdots +  (h'_{d-1}+h''_{d-2}) t^{d-1}+ (h'_d+h''_{d-1}) t^d$$
by \eqref{Hilb (n-d,d)}.  Now, we have 
$$h'_1+1 = (n-d-1)+1 =n-d,$$ 
$$h'_i+h''_{i-1}= \binom{n-d+i-2}{i} + \binom{n-d+i-2}{i-1} =  \binom{n-d+i-1}{i}$$
for $2 \le i \le d-1$, and 
$$h'_d+h''_{d-1}= \binom{n-2}{d-2}  + \binom{n-2}{d-3}  =  \binom{n-1}{d-2}.$$
So the assertion holds in this case.

Next, assuming  that  $n=2d$, we consider the Hilbert series of $R/\ISp_{(d,d)}$ with $d \ge 2$. By the induction hypothesis, $H(S/\ISp_{(d,d-1)}, t)$ does not depend on $\chara(K)$, and we have  
$$H(S/\ISp_{(d,d-1)}, t)= \frac{1+h''_1 t +h''_2 t^2+\cdots + h''_{d-1} t^{d-1} }{(1-t)^{d-1}}$$
with 
$$
h''_i = \begin{cases}
\binom{d+i-1}{i} & \text{if $1 \le i \le d-2$,} \\
\binom{2d-2}{d-3} &  \text{if $ i = d-1$.}
\end{cases}
$$
We have already shown that
$$H(S/\ISp_{(d-1,d-1,1)}, t)= \frac{1+h'_1 t +h'_2 t^2+\cdots + h'_d t^d }{(1-t)^d}$$
with 
$$
h'_i = \binom{d+i-2}{i}.
$$
Since $\dim R/\ISp_{(d,d)}=d$, the denominator of the Hilbert series of  $H(R/\ISp_{(d,d)}, t)$ is $(1- t)^d$, and the numerator is 
$$1+(h'_1+1)t+ (h'_2+h''_1)t^2+\cdots +  (h'_{d-1}+h''_{d-2}) t^{d-1}+ (h'_d+h''_{d-1})t^d$$
by \eqref{Hilb (d,d)}.  Now, we have 
$$h'_1+1 = (d-1)+1 =d,$$ 
$$h'_i+h''_{i-1}= \binom{d+i-2}{i} + \binom{d+i-2}{i-1} =  \binom{d+i-1}{i}$$
for $2 \le i \le d-1$, and 
$$h'_d+h''_{d-1}= \binom{2d-2}{d}  + \binom{2d-2}{d-3}  =   \binom{2d-2}{d-2}  + \binom{2d-2}{d-3}  = \binom{2d-1}{d-2}.$$
So we are done.  

\qed


\begin{thebibliography}{8}
\bibitem{E}
D. Eisenbud, Commutative Algebra with a view toward algebraic geometry, Graduate texts in Mathematics 150, Springer-Verlag, 1995.  

\bibitem{EGL} P. Etingof, E. Gorsky, and I. Losev, Representations of rational Cherednik algebras with minimal support and torus knots,  Adv. Math.  {\bf 277} (2015), 124--180. 

\bibitem{RV} C. Renter\'ia and R. H. Villarreal, Koszul homology of Cohen-Macaulay rings with linear resolutions,  Proc. Amer. Math. Soc.  {\bf 115} (1992), 51--58. 

\bibitem{Sa} B.E. Sagan, The Symmetric Group. Representations, Combinatorial Algorithms, and Symmetric Functions, second edition, Graduate texts in Mathematics 203, Springer-Verlag,  2001.



\bibitem{V} R.H.~Villarreal. Monomial Algebras, second edition, Monographs and Research Notes in Mathematics, Chapman and Hall/CRC Press, Boca Raton, FL, 2015.

\bibitem{WY} J. Watanabe and K. Yanagawa, Vandermonde determinantal ideals, Math. Scand. {\bf 125} (2019),  179-184. 

\bibitem{Y} K. Yanagawa, When is a Specht ideal Cohen-Macaulay?   J. Commut. Algebra (to appear), arXiv:1902.06577.  
\end{thebibliography}
\end{document}